\documentclass[leqno,12pt]{amsart}
\usepackage{amsmath}
\usepackage{amsfonts}
\usepackage{amssymb}
\usepackage{amsthm}
\usepackage[usenames,dvipsnames]{color}
\usepackage{hyperref}
\newtheorem{thm}{Theorem}
\newtheorem{cor}[thm]{Corollary}
\numberwithin{equation}{section}
\numberwithin{thm}{section}
\newcommand{\set}[1]{\left\{{#1}\right\}}
\newcommand{\abs}[1]{\left\vert{#1}\right\vert}
\newcommand{\Z}{\mathbb{Z}}
\newcommand{\Q}{\mathbb{Q}}
\newcommand{\R}{\mathbb{R}}
\long\def\blankfootnotetext#1{\begingroup\def\thefootnote{\fnsymbol{footnote}}\footnotetext{#1}\endgroup}
\begin{document}
\title{A note on palindromic $\delta$-vectors for certain rational polytopes}
\author{Matthew H.\ J.\ Fiset and Alexander M.\ Kasprzyk}
\address{Department of Mathematics and Statistics\\University of New Brunswick\\Fredericton NB\\E3B 5A3\\Canada.}
\email{u0a35@unb.ca\textrm{,} kasprzyk@unb.ca}
\maketitle
\blankfootnotetext{The first author was funded by an NSERC USRA grant. The second author is funded by an ACEnet research fellowship.}
\blankfootnotetext{2000 \textit{Mathematics Subject Classification}. Primary 05A15; Secondary 11H06.}
\blankfootnotetext{\textit{Key words and phrases}. Rational polytope, Ehrhart $\delta$-vector, palindromic theorem.}
\begin{abstract}
Let $P$ be a convex polytope containing the origin, whose dual is a lattice polytope. Hibi's Palindromic Theorem tells us that if $P$ is also a lattice polytope then the Ehrhart $\delta$-vector of $P$ is palindromic. Perhaps less well-known is that a similar result holds when $P$ is rational. We present an elementary lattice-point proof of this fact.
\end{abstract}
\section{Introduction}
A \emph{rational polytope} $P\subset \R^n$ is the convex hull of finitely many points in $\Q^n$. We shall assume that $P$ is of maximum dimension, so that $\dim{P}=n$. Throughout let $k$ denote the smallest positive integer for which the dilation $kP$ of $P$ is a \emph{lattice polytope} (i.e.~the vertices of $kP$ lie in $\Z^n$).

A \emph{quasi-polynomial} is a function defined on $\Z$ of the form:
$$q(m)=c_n(m)m^n+c_{n-1}(m)m^{n-1}+\ldots+c_0(m),$$
where the $c_i$ are periodic coefficient functions in $m$. It is known~(\cite{Ehr62}) that for a rational polytope $P$, the number  of lattice points in $mP$, where $m\in\Z_{\geq 0}$, is given by a quasi-polynomial of degree $n=\dim{P}$ called the \emph{Ehrhart quasi-polynomial}; we denote this by $L_P(m):=\abs{mP\cap \Z^n}$. The minimum period common to the cyclic coefficients $c_i$ of $L_P$ divides $k$ (for further details see~\cite{BSW08}).

Stanley proved in~\cite{Stan80} that the generating function for $L_P$ can be written as a rational function:
$$\mathrm{Ehr}_P(t):=\sum_{m\geq 0}L_P(m)t^m=\frac{\delta_0+\delta_1t+\ldots+\delta_{k(n+1)-1}t^{k(n+1)-1}}{(1-t^k)^{n+1}},$$
whose coefficients $\delta_i$ are non-negative. For an elementary proof of this and other relevant results, see~\cite{BS07} and~\cite{BR}. We call $(\delta_0,\delta_1,\ldots,\delta_{k(n+1)-1})$ the \emph{(Ehrhart) $\delta$-vector} of $P$.

The \emph{dual polyhedron} of $P$ is given by $P^\vee:=\set{u\in\R^n\mid\left<u,v\right>\leq 1\text{ for all }v\in P}$. If the origin lies in the interior of $P$ then $P^\vee$ is a rational polytope containing the origin, and $P=(P^\vee)^\vee$. We restrict our attention to those $P$ containing the origin for which $P^\vee$ is a lattice polytope.

We give an elementary lattice-point proof that, with the above restriction, the $\delta$-vector is palindromic (i.e. $\delta_i=\delta_{k(n+1)-1-i}$). When $P$ is \emph{reflexive}, meaning that $P$ is also a lattice polytope (equivalently, $k=1$), this result is known as \emph{Hibi's Palindromic Theorem}~\cite{Hib91}. It can be regarded as a consequence of a theorem of Stanley's concerning the more general theory of Gorenstein rings; see~\cite{Sta78}.
\section{The main result}
Let $P$ be a rational polytope and consider the Ehrhart quasi-polynomial $L_P$. There exist $k$ polynomials $L_{P,r}$ of degree $n$ in $l$ such that when $m=lk+r$ (where $l,r\in\Z_{\geq 0}$ and $0\leq r<k$) we have that $L_P(m)=L_{P,r}(l)$. The generating function for each $L_{P,r}$ is given by:
\begin{equation}
\mathrm{Ehr}_{P,r}(t):=\sum_{l\geq 0}L_{P,r}(l)t^l=\frac{\delta_{0,r}+\delta_{1,r}t+\ldots+\delta_{n,r}t^n}{(1-t)^{n+1}},
\label{eq:partial_Ehr}
\end{equation}
for some $\delta_{i,r}\in\Z$.

\begin{thm}\label{thm:main_result}
Let $P$ be a rational $n$-tope containing the origin, whose dual $P^\vee$ is a lattice polytope. Let $k$ be the smallest positive integer such that $kP$ is a lattice polytope. Then:
$$\delta_{i,r}=\delta_{n-i,k-r-1}.$$
\end{thm}
\begin{proof}
By Ehrhart--Macdonald reciprocity (\cite{Ehr67,Mac71}) we have that:
$$L_P(-lk-r)=(-1)^nL_{P^\circ}(lk+r),$$
where $L_{P^\circ}$ enumerates lattice points in the strict interior of dilations of $P$. The left-hand side equals $L_P\left(-(l+1)k+(k-r)\right)=L_{P,k-r}\left(-(l+1)\right)$. We shall show that the right-hand side is equal to $(-1)^nL_P(lk+r-1)=(-1)^nL_{P,r-1}(l)$.

Let $H_u:=\set{v\in\R^n\mid\left<u,v\right>=1}$ be a bounding hyperplane of $P$, where $u\in\mathrm{vert}\,P^\vee$. By assumption, $u\in\Z^n$ and so the lattice points in $\Z^n$ lie at integer heights relative to $H_u$; i.e. given $u'\in\Z^n$ there exists some $c\in\Z$ such that $u'\in\set{v\in\R^n\mid\left<u,v\right>=c}$. In particular, there do not exist lattice points at non-integral heights. Since:
$$P=\bigcap_{u\in\mathrm{vert}\,P^\vee}H^-_u,$$
where $H^-_u$ is the half-space defined by $H_u$ and the origin, we see that $(mP^\circ)\cap\Z^n=\left((m-1)P\right)\cap\Z^n$. This gives us the desired equality.

We have that $L_{P,k-r}\left(-(l+1)\right)=(-1)^nL_{P,r-1}(l)$. By considering the expansion of~(\ref{eq:partial_Ehr}) we obtain:
\begin{align*}
\sum_{i=0}^n\delta_{i,k-r}{-(l+1)+n-i\choose{n}}&=L_{P,k-r}(-(l+1))\\
&=(-1)^nL_{P,r-1}(l)=(-1)^n\sum_{i=0}^n\delta_{i,r-1}{l+n-i\choose{n}}.
\end{align*}
But ${-(l+1)+n-i\choose{n}}=(-1)^n{l+n-i\choose{n}}$, and since ${l\choose{n}},{l+1\choose{n}},\ldots,{l+n\choose{n}}$ form a basis for the vector space of polynomials in $l$ of degree at most $n$, we have that $\delta_{i,k-r}=\delta_{n-i,r-1}$.
\end{proof}

\begin{cor}
The $\delta$-vector of $P$ is palindromic.
\end{cor}
\begin{proof}
This is immediate once we observe that:
$$\mathrm{Ehr}_P(t)=\mathrm{Ehr}_{P,0}(t^k)+t\mathrm{Ehr}_{P,1}(t^k)+\ldots+t^{k-1}\mathrm{Ehr}_{P,k-1}(t^k).$$
\end{proof}
\section{Concluding remarks}
The crucial observation in the proof of Theorem~\ref{thm:main_result} is that $(mP^\circ)\cap\Z^n=\left((m-1)P\right)\cap\Z^n$. In fact, a consequence of Ehrhart--Macdonald reciprocity and a result of Hibi~\cite{Hibi92} tells us that this property holds if and only if $P^\vee$ is a lattice polytope. Hence rational convex polytopes whose duals are lattice polytopes are characterised by having palindromic $\delta$-vectors. This can also be derived from Stanley's work~\cite{Sta78} on Gorenstein rings.
\bibliographystyle{amsalpha}
\providecommand{\bysame}{\leavevmode\hbox to3em{\hrulefill}\thinspace}
\providecommand{\MR}{\relax\ifhmode\unskip\space\fi MR }
\providecommand{\MRhref}[2]{%
  \href{http://www.ams.org/mathscinet-getitem?mr=#1}{#2}
}
\providecommand{\href}[2]{#2}

\end{document}